\documentclass[11pt, oneside]{article}   	

\usepackage{geometry}                		
\usepackage{graphicx}				
\usepackage{amssymb}
\usepackage{amsthm}

\newtheorem{thm}{Theorem}[section]
\newtheorem{lemma}[thm]{Lemma}
\newtheorem{prop}[thm]{Proposition}
\newtheorem{cor}[thm]{Corollary}

\newtheorem{claim}[thm]{Claim}

\theoremstyle{definition}
\newtheorem{define}[thm]{Definition}

\usepackage{tikz}
\usetikzlibrary{calc}

\begin{document}

\title{Exponentially Many 4-List-Colorings of Triangle-Free Graphs on Surfaces}
\author{Tom Kelly\thanks{Department of Combinatorics and Optimization, University of Waterloo, 200 University Ave West, Waterloo, Ontario, Canada N2L 3G1.}
\thanks{Email: \texttt{t9kelly@uwaterloo.ca}}
\and
Luke Postle\footnotemark[1]
\thanks{Partially supported by NSERC under Discovery Grant No. 2014-06162.  Email: \texttt{lpostle@uwaterloo.ca}}
}

\date{February 15, 2016}
 \maketitle

\begin{abstract}
Thomassen proved that every planar graph $G$ on $n$ vertices has at least $2^{n/9}$ distinct $L$-colorings if $L$ is a 5-list-assignment for $G$ and at least $2^{n/10000}$ distinct $L$-colorings if $L$ is a 3-list-assignment for $G$ and $G$ has girth at least five.  Postle and Thomas proved that if $G$ is a graph on $n$ vertices embedded on a surface $\Sigma$ of genus $g$, then there exist constants $\epsilon,c_g > 0$ such that if $G$ has an $L$-coloring, then $G$ has at least $c_g2^{\epsilon n}$ distinct $L$-colorings if $L$ is a 5-list-assignment for $G$ or if $L$ is a 3-list-assignment for $G$ and $G$ has girth at least five.  More generally, they proved that there exist constants $\epsilon,\alpha>0$ such that if $G$ is a graph on $n$ vertices embedded in a surface $\Sigma$ of fixed genus $g$, $H$ is a proper subgraph of $G$, and $\phi$ is an $L$-coloring of $H$ that extends to an $L$-coloring of $G$, then $\phi$ extends to at least $2^{\epsilon(n - \alpha(g + |V(H)|))}$ distinct $L$-colorings of $G$ if $L$ is a 5-list-assignment or if $L$ is a 3-list-assignment and $G$ has girth at least five.  We prove the same result if $G$ is triangle-free and $L$ is a 4-list-assignment of $G$, where $\epsilon=\frac{1}{8}$, and $\alpha= 130$.
\end{abstract}

\section{Introduction}

Let $G$ be a graph with $n$ vertices, and let $L = (L(v) : v\in V(G))$ be a collection of lists which we call {\em available colors}.  If each set $L(v)$ is non-empty, then we say that $L$ is a {\em list-assignment} for $G$.  If $k$ is an integer and $|L(v)|\geq k$ for every $v\in V(G)$, then we say that $L$ is a {\em $k$-list-assignment} for $G$.  An {\em $L$-coloring} of $G$ is a mapping $\phi$ with domain $V(G)$ such that $\phi(v)\in L(v)$ for every $v\in V(G)$ and $\phi(v)\neq\phi(u)$ for every pair of adjacent vertices $u,v\in V(G)$.  We say that a graph $G$ is {\em $k$-choosable}, or {\em $k$-list-colorable}, if $G$ has an $L$-coloring for every $k$-list-assignment $L$.  If $L(v) = \{1,\dots, k\}$ for every $v\in V(G)$, then we call an $L$-coloring of $G$ a {\em $k$-coloring}, and we say $G$ is {\em $k$-colorable} if $G$ has a $k$-coloring.

If $G$ has an $L$-coloring, it is natural to ask how many $L$-colorings $G$ has.  In particular, we are interested in when the number of $L$-colorings of $G$ is exponential in the number of vertices.  The Four Color Theorem states that every planar graph has a 4-coloring.  A plane graph obtained from the triangle by recursively adding vertices of degree three inside facial triangles has only one 4-coloring up to permutation of the colors.  So in general planar graphs do not have exponentially many 4-colorings.  However, if $\phi$ is a $k$-coloring of $G$, then we may assume there is some $X\subseteq V(G)$ with $|X|\geq |V(G)|/k$ such that for all $v\in X$, $\phi(v)=1$.  It follows that $G$ has at least $2^{|V(G)|/k}$ $(k+1)$-colorings, because for each subset of $X$, we can obtain a unique $(k+1)$-coloring of $G$ from $\phi$ by coloring it with the color $k+1$.  Hence, planar graphs have exponentially many 5-colorings.  In \cite{BL46}, Birkhoff and Lewis obtained an optimal bound on the number of 5-colorings of planar graphs, which is tight for the graph described above.

\begin{thm}\cite{BL46}
Every planar graph on $n \geq 3$ vertices has at least $60\cdot 2^{n-3}$ distinct 5-colorings
\end{thm}

In \cite{Th06}, Thomassen proved a similar result for graphs on surfaces.

\begin{thm}\label{surface coloring}\cite{Th06}
For every surface $\Sigma$ there is some constant $c>0$ such that every 5-colorable graph on $n$ vertices embedded in $\Sigma$ has at least $c\cdot 2^n$ distinct 5-colorings.\end{thm}

In \cite[Theorem 2.1]{Th06}, Thomassen gave a shorter proof using Euler's formula that for every fixed surface $\Sigma$, if a graph $G$ embedded in $\Sigma$ is 5-colorable, then it has exponentially many 5-colorings.  The argument also applies to 4-colorings of triangle-free graphs and 3-colorings of graphs of girth at least five.  We are interested in finding similar results for list-coloring.

In \cite{Th94}, Thomassen gave his classic proof that every planar graph is 5-choosable.  Later, Thomassen proved that in fact every planar graph has exponentially many 5-list-colorings.

\begin{thm}\label{planar list 5}\cite{Th07a}
If $G$ is a planar graph on $n$ vertices and $L$ is a 5-list-assignment for $G$, then $G$ has at least $2^{n/9}$ distinct $L$-colorings.\end{thm}

In \cite{Th03}, Thomassen proved that every planar graph of girth at least five is 3-choosable.  Later, he proved that in fact every planar graph of girth at least 5 has exponentially many 3-list-colorings.

\begin{thm}\label{planar list 3}\cite{Th07b}
If $G$ is a planar graph on $n$ vertices of girth at least 5 and $L$ is a 3-list-assignment for $G$, then $G$ has at least $2^{n/10000}$ distinct $L$-colorings.\end{thm}

An important proof technique is to extend a coloring of a subgraph to the entire graph.  This can be viewed as list-coloring where the precolored vertices have lists of size one.  The following theorem of Postle and Thomas \cite{PT16, P16} utilizes this technique and extends Theorems \ref{planar list 5} and \ref{planar list 3} to graphs on surfaces.

\begin{thm}\label{surface lists}\cite{PT16,P16}
There exist constants $\epsilon,\alpha > 0$ such that the following holds.  Let $G$ be a graph on $n$ vertices embedded in a fixed surface $\Sigma$ of genus $g$, and let $H$ be a proper subgraph of $G$.  If $L$ is a 5-list-assignment for $G$, or $L$ is a 3-list-assignment for $G$ and $G$ has girth at least five, and if $\phi$ is an $L$-coloring of $H$ that extends to an $L$-coloring of $G$, then $\phi$ extends to at least $2^{\epsilon(n - \alpha(g + |V(H)|))}$ distinct $L$-colorings of $G$.\end{thm}

A classical theorem of Gr\H otzsch states that every triangle-free planar graph is 3-colorable.  Hence, every triangle-free planar graph has exponentially many 4-colorings.  Thomassen conjectured in \cite{Th07b} that in fact every triangle-free planar graph has exponentially many 3-colorings.  The best progress towards this conjecture is the following result due to Asadi et al..

\begin{thm}\cite{ADPT13}\label{sub-exp}
Every triangle-free planar graph on $n$ vertices has at least $2^{\sqrt{n/212}}$ distinct 3-colorings.
\end{thm}

Theorem \ref{sub-exp} can not be extended to list-coloring, since there exist triangle-free planar graphs that are not 3-choosable.  However, it is an easy consequence of Euler's formula that every triangle-free planar graph is 4-choosable.  Thus, it is natural to ask if a result analagous to Theorem \ref{surface lists} holds for 4-list-coloring triangle-free graphs on surfaces. The following is our main theorem.

\begin{thm}\label{main thm}
Let $G$ be a triangle-free graph on $n$ vertices embedded in a fixed surface $\Sigma$ of genus $g$, and let $L$ be a 4-list-assignment for $G$.  If $H\subsetneq G$, and $\phi$ is an $L$-coloring of $H$ that extends to $G$, then $\phi$ extends to $2^{(n - 130(g + |V(H)|))/8}$ distinct $L$-colorings of $G$.
\end{thm}

In order to prove Theorem \ref{main thm}, we prove a stronger result for which we need the following definition.
\begin{define}
Let $\epsilon,\alpha\geq 0$.  Let $G$ be a graph embedded in a surface $\Sigma$ of Euler genus $g$, let $H$ be a proper subgraph of $G$, and let $L$ be a list-assignment for $G$.  We say that $(G, H)$ is {\em $(\epsilon,\alpha)$-exponentially-critical} with respect to $L$ if for every proper subgraph $G'$ of $G$ such that $H\subseteq G'$, there exists an $L$-coloring $\phi$ of $H$ such that there exists $2^{\epsilon(|V(G')| - \alpha(g + |V(H)|))}$ distinct $L$-colorings of $G'$ extending $\phi$, but there do not exist $2^{\epsilon(|V(G)| - \alpha(g+|V(H)|))}$ distinct $L$-colorings of $G$ extending $\phi$.
\end{define}

We prove the following theorem, which implies Theorem \ref{main thm}.

\begin{thm}\label{real thm}
Suppose $(G,H)$ is $(\epsilon,\alpha)$-exponentially-critical and $G$ is triangle-free.  For all $\alpha\geq0$, if $0\leq\epsilon\leq\frac{1}{8}$, then $|V(G)|\leq 50\left(|V(H)|-\frac{13}{5}\right) + 130g$.
\end{thm}

\begin{proof}[Proof of Theorem \ref{main thm} assuming Theorem \ref{real thm}]
Let $(G,H)$ be a minimal counterexample.  Then there exists an $L$-coloring $\phi$ of $H$ that extends to $G$ that does not extend to $2^{(V(G)| - 130(g + |V(H)|))/8}$ distinct $L$-colorings of $G$.  By the minimality of $G$, $G$ is $(\epsilon,\alpha)$-exponentially-critical, where $\epsilon=\frac{1}{8}$ and $\alpha = 130$.  Hence, by Theorem \ref{real thm}, $|V(G)| \leq 50(|V(H)| - \frac{13}{5}) + 130g$.  Therefore $\phi$ does not extend to an $L$-coloring of $G$, a contradiction.
\end{proof}

We prove Theorem \ref{real thm} using the method of reducible configurations and discharging.  In this paper, if $G$ is a graph and $H\subsetneq G$, then a {\em reducible configuration} of $(G,H)$ is a nonempty subgraph $Q$ of $G-V(H)$ such that for every 4-list-assignment $L$ of $G$, every $L$-coloring of $G-V(Q)$ extends to at least two distinct $L$-coloring of $G$.  In Section \ref{reducible config}, we prove that certain reducible configurations do not occur in $(\epsilon,\alpha)$-exponentially-critical graphs.  In Section \ref{proof section}, we prove Theorem \ref{real thm} using discharging.

Finally, we remark that a version of Theorem \ref{real thm} can be proved if $\epsilon \leq \frac{1}{7}$, at the expense of a worse bound on $|V(G)|$ and a more complicated discharging argument.  


\section{Reducible Configurations}\label{reducible config}

We first prove that small reducible configurations do not occur in $(\epsilon,\alpha)$-exponentially-critical graphs.  

\begin{prop}\label{no reducible configurations}
If $(G,H)$ is $(\epsilon,\alpha)$-exponentially-critical with respect to some 4-list-assignment $L$, then $(G,H)$ does not contain any reducible configurations of size at most $\frac{1}{\epsilon}$.
\end{prop}
\begin{proof}
Suppose that $Q\subseteq G-V(H)$ is a reducible configuration.  We want to show $|V(Q)| > \frac{1}{\epsilon}$.  Since $(G,H)$ is $(\epsilon,\alpha)$-exponentially-critical, there exists an $L$-coloring $\phi$ of $H$ such that there exists $2^{\epsilon(|V(G)| - |V(Q)| - \alpha(g + |V(H)|))}$ distinct $L$-colorings of $G-V(Q)$ extending $\phi$, but there do not exist $2^{\epsilon(|V(G)| - \alpha(g + |V(H)|))}$ distinct $L$-colorings of $G$ extending $\phi$.  Since $Q$ is a reducible configuration, every $L$-coloring of $G-V(Q)$ extending $\phi$ has at least two extensions to an $L$-coloring of $G$.  Hence, $G$ has at least $2^{\epsilon(|V(G)| - |V(Q| - \alpha(g + |V(H)|)) + 1} = 2^{\epsilon(|V(G)| - \alpha(g + |V(H)|)) + 1 - \epsilon|V(Q)|}$ distinct $L$-colorings extending $\phi$.  Therefore $|V(Q)| > \frac{1}{\epsilon}$, as desired.
\end{proof}

We now present our first reducible configuration.

\begin{lemma}\label{4-face with 4-vertices}A 4-cycle $C\subseteq G-V(H)$ is a reducible configuration if for all $v\in V(C)$, $v$ has degree at most four in $G$.\end{lemma}
\begin{proof}
Let $L$ be some 4-list-assignment for $G$, and let $\phi$ be an $L$-coloring of $G-V(C)$.  Note that there are two distinct list-colorings of a 4-cycle when every vertex has at least two available colors.  Hence, there are at least two distinct $L$-colorings of $G$ extending $\phi$, as desired.
\end{proof}

For our next reducible configuration, we need the following definitions.  

\begin{define}If $P$ is a path, and $v\in V(P)$ is not an end of $P$, then we say $v$ is an {\em internal vertex} of $P$.  If $P'$ is also a path, we say $P$ and $P'$ are {\em internally disjoint} if they share no internal vertices.\end{define}

\begin{define}We say a path $P\subseteq G$ is a {\em stamen} in $(G, H)$ if there exists an end $u\in V(G)\backslash V(H)$ of $P$ such that the degree of $u$ is precisely three in $G$, and in addition, every internal vertex of $P$ has degree four and is not in $H$.  If $v\neq u$ is an end of $P$, then we say $P$ is a {\em $v$-stamen}.\end{define}

If $v\in V(G)$, let $d(v)$ denote the degree of $v$ in $G$.

\begin{define}We say $G'\subseteq G-V(H)$ is a {\em poppy} of $(G, H)$ if there is some $v\in V(G')$ such that $G'$ is the union of $v$ and at least $d(v)-2$ internally disjoint $v$-stamens.\end{define}

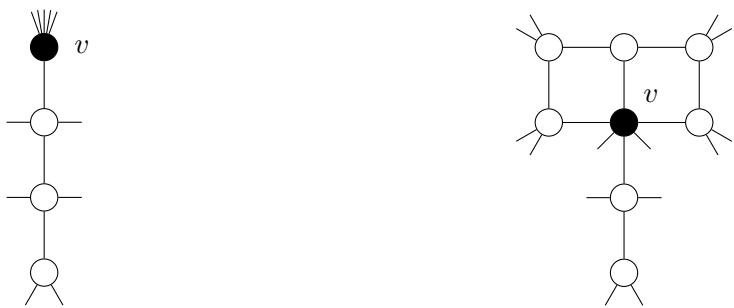
\begin{figure}[h]
\begin{minipage}[b]{.5\linewidth}
	\centering
	\begin{tikzpicture}
		\tikzstyle{every node}=[draw, circle];
		
		\node[label=0:$v$, fill] (1) at (0, 0) {};
		\node(2) at (0,-1) {};
		\node(3) at (0, -2) {};
		\node (4) at (0, -3) {};
		
		\draw (1) -- (2) -- (3) -- (4);
		
		\draw (1) -- ($(1) + (70:.5)$);
		\draw (1) -- ($(1) + (80:.5)$);
		\draw (1) -- ($(1) + (90:.5)$);
		\draw (1) -- ($(1) + (100:.5)$);
		\draw (1) -- ($(1) + (110:.5)$);
		
		\draw (2) -- ($(2) + (0:.5)$); \draw (2) -- ($(2) + (180:.5)$);
		\draw (3) -- ($(3) + (0:.5)$); \draw (3) -- ($(3) + (180:.5)$);
		\draw (4) -- ($(4) + (-60:.5)$); \draw (4) -- ($(4) + (-120:.5)$);
		
	\end{tikzpicture}
\end{minipage}%
\begin{minipage}[b]{.5\linewidth}
	\centering
	\begin{tikzpicture}
		\tikzstyle{every node}=[draw, circle];
		
		\node[label=45:$v$, fill] (main) at (0, 0) {};
		\node (1) at (1, 0) {};
		\node (2) at (1, 1) {};
		\node (3) at (-1, 0) {};
		\node (4) at (-1, 1) {};
		\node (5) at (0, 1) {};
		
		\draw (main) -- (1) -- (2) -- (5) -- (4) -- (3) -- (main) -- (5);
		
		\draw (main) -- ($(main) + (-45:.5)$); \draw (main) -- ($(main) + (-135:.5)$);
		
		\draw (1) -- ($(1) + (-60:.5)$); \draw (1) -- ($(1) + (-30:.5)$);
		\draw (2) -- ($(2) + (60:.5)$); \draw (2) -- ($(2) + (30:.5)$);
		
		\draw (3) -- ($(3) + (-120:.5)$); \draw (3) -- ($(3) + (-150:.5)$);
		\draw (4) -- ($(4) + (120:.5)$); \draw (4) -- ($(4) + (150:.5)$);
		
		\node (6) at (0, -1){}; \draw (main) -- (6);
		\node (7) at (0,-2){}; \draw (6) -- (7);
		\draw (6) -- ($(6) + (0:.5)$); \draw (6) -- ($(6) + (180:.5)$);
		\draw (7) -- ($(7) + (-60:.5)$); \draw (7) -- ($(7) + (-120:.5)$);
	\end{tikzpicture}
\end{minipage}%
\caption{A $v$-stamen and a poppy}
\label{6-vertex triple}
\end{figure}

We next prove that a poppy is a reducible configuration, but first we need the following definition and a classical theorem of Erd\H os, Rubin, and Taylor \cite{ERT80}.
\begin{define}We say $G$ is {\em degree-choosable} if for every list-assignment $L$ such that for all $v\in V(G)$, $|L(v)|\geq d(v)$, $G$ has an $L$-coloring.\end{define}

\begin{thm}\cite{ERT80}\label{degree-choosable}
A connected graph $G$ is not degree-choosable if and only if every block of $G$ is a clique or an odd cycle.  Furthermore, if $G$ does not have an $L$-coloring for some $L$ with $|L(v)|\geq d(v)$, then for all $v\in V(G)$, $|L(v)| = d(v)$.\end{thm}

\begin{lemma}\label{super lemma}
If $Q$ is a poppy of $(G,H)$, then $Q$ is a reducible configuration.
\end{lemma}
\begin{proof}
Let $Q$ be a poppy of $(G,H)$.  Let $L$ be some 4-list-assignment of $G$, and let $\phi$ be an $L$-coloring of $G-V(Q)$.   Say $Q$ is the union of $v$ and $v$-stamens $P_1,\dots, P_k$, where $k \geq d(v)-2$.  Let $L'$ be a list-assignment for $Q$, where for every $u\in V(Q)$, $L'(u) = L(u)\backslash\{\phi(u') : uu'\in E(G), u'\in V(G)\backslash V(Q-v)\}$.  Let $\phi_1=\phi_2=\phi$, and let $\phi_1(v)\neq\phi_2(v)\in L'(v)$.  

Note that every connected component of $Q-v$ contains a vertex $u$ of degree three in $G$, so $|L'(u)| = d_{Q-v}(u)+1$.  Therefore by Theorem \ref{degree-choosable}, every connected component of $Q-v$ is $L'$-colorable.  Hence, $\phi_1$ and $\phi_2$ extend to distinct $L$-colorings of $G$, so $Q$ is a reducible configuration, as desired.
\end{proof}

If $v\in V(G)$ has degree at most two, then $v$ itself is a poppy.  Hence, Lemma \ref{super lemma} implies the following.

\begin{cor}\label{min degree 3}
If $v\in V(G)$ has degree at most two, then $v$ is a reducible configuration.
\end{cor}

If $v\in V(G)$ has degree three, then a $v$-stamen in $(G,H)$ is a poppy.  Hence, Lemma \ref{super lemma} implies the following.
\begin{cor}\label{path bt 3-vertices}If $v\in V(G)\backslash V(H)$ has degree three, then a $v$-stamen is a reducible configuration of $(G,H)$.\end{cor}


\section{Discharging}\label{proof section}

Before proving Theorem \ref{real thm}, we need some definitions.  In the following definitions, $G$ is a graph and $H\subsetneq G$.  

\begin{define}We say $v\in V(G)$ is a {\em $k$-vertex} if $d(v)=k$, a {\em $k^+$-vertex} if $d(v)\geq k$, and a {\em $k^-$-vertex} if $d(v)\leq k$.  If $G$ is embedded in a surface, we define a {\em $k$-face}, a {\em $k^+$-face}, and a {\em $k^-$-face} similarly.\end{define}

\begin{define}We say $v\in V(G)$ is a {\em major vertex} of $(G,H)$ if $v$ is a $5^+$-vertex, or if $v\in V(H)$.\end{define}

\begin{define}
If every vertex of a stamen $P$ of $G$ is incident with a face $f$, then we say $P$ is {\em incident with} $f$.
\end{define}

\begin{define}If $G$ is 2-cell-embedded in some surface $\Sigma$ and $f$ is a face of $G$, then the boundary of $f$ in $\Sigma$ is the union of the vertices and edges of a closed walk in $G$, which we call the {\em boundary walk} of $f$.\end{define}

If $G$ is embedded in a surface, we let $F(G)$ denote the set of faces of $G$.  If $G$ is 2-cell-embedded and $f\in F(G)$, we let $|f|$ denote the length of the boundary walk of $f$.  We are now ready to prove Theorem \ref{real thm}.

\begin{proof}[Proof of Theorem \ref{real thm}]
Suppose $G$ is a triangle-free graph embedded in a surface $\Sigma$ of Euler genus $g$, $H\subsetneq G$, and $(G,H)$ is $(\epsilon,\alpha)$-exponentially-critical with respect to some 4-list-assignment $L$, where $0\leq\epsilon\leq\frac{1}{8}$.  Let $G_1,\dots, G_m$ be the components of $G$, and let $H_i = G_i\cap H$.  
To prove Theorem \ref{real thm}, it suffices to show that for all $i=1,\dots, m$, $|V(G_i)| \leq 50(|V(H_i)| - \frac{13}{5}) + 130g_i$ when $V(H_i)\subsetneq V(G_i)$ and $g_i$ is the genus of $G_i$.  

By Proposition \ref{no reducible configurations}, $(G,H)$ has no reducible configurations of size at most $\frac{1}{\epsilon}$.  Note that a reducible configuration of $(G_i,H_i)$ is a reducible configuration of $(G,H)$.  Thus, for all $i=1,\dots, m$, $(G_i,H_i)$ has no reducible configurations of size at most $\frac{1}{\epsilon}$.
Hence, it suffices to show $|V(G)| \leq 50(|V(H)| - \frac{13}{5}) + 130g$, where $G$ is a connected triangle-free graph embedded in a surface $\Sigma$ of Euler genus $g$, $H\subsetneq G$, and $(G,H)$ contains no reducible configurations of size at most $\frac{1}{\epsilon}$.  We may assume $G$ is 2-cell-embedded in $\Sigma$, or else we embed $G$ in a surface of smaller genus.

For $v\in V(G)\backslash V(H)$, let $ch(v) = d(v) - 4$, and for $v\in V(H)$, let $ch(v) = d(v) + 3\gamma - 1$ for some fixed constant $\gamma>0$ to be determined later.  For every $f\in F(G)$, let $ch(f) = |f| - 4$.  By Euler's formula,
$$\sum_{v\in V(G)}ch(v) + \sum_{f\in F(G)}ch(f) = (3 + 3\gamma)|V(H)|+ 4(2g - 2).$$

Redistribute the charges according to the following rules, and let $ch_*$ denote the final charge.  

\begin{enumerate}
	\item Let $v$ be a major vertex, and let $u\in V(G)\backslash V(H)$ be a 3-vertex at distance at most two from $v$.  For every $v$-stamen $P$ in $G$ with an end at $u$ such that there exists a 4-face $f$ with $P$ incident with $f$, let $v$ send charge $\frac{1}{3} + \gamma$ to $u$.
	\item Let $v$ be a major vertex, and let $u\in V(G)\backslash V(H)$ be a 4-vertex at distance at most two from $v$.  For each 4-face incident to both $u$ and $v$, let $v$ send charge $\frac{3\gamma}{4}$ to $u$.
	\item If $f$ is a $5^+$-face incident to a 3-vertex $u\in V(G)\backslash V(H)$, let $f$ send charge $\frac{1}{3} + \gamma$ to $u$ for every instance of $u$ in the boundary walk of $f$.
	\item If $f$ is a $5^+$-face incident to a 4-vertex $u\in V(G)\backslash V(H)$, let $f$ send charge $\frac{3\gamma}{4}$ to $u$ for every instance of $u$ in the boundary walk of $f$.
\end{enumerate}

Figure \ref{rule 1} illustrates an instance of Rule 1.  Major vertices are represented as black circles, and non-major vertices are represented as white circles.  There are two $v$-stamens and one $v'$-stamen with ends at $u$ (shown as directed paths), and each is incident with a 4-face.  Hence, $v$ sends charge at least $\frac{2}{3} + 2\gamma$ to $u$ and $v'$ sends charge at least $\frac{1}{3} + \gamma$ to $u$ under Rule 1.

\begin{figure}
	\centering
	\begin{tikzpicture}[scale=1.5]
		\tikzstyle{every node}=[draw, circle];
		
		\node[label=45:$v$, fill] (main) at (0, 0) {};
		\node[label=45:$v'$, fill] (1) at (1, 0) {};
		\node (2) at (1, 1) {};
		\node (3) at (-1, 0) {};
		\node (4) at (-1, 1) {};
		\node[label=45:$u$] (5) at (0, 1) {};
		
		\draw [->] (main) edge (5); \draw (main) -- (1);
		\draw [->] (1) edge (2) (2) edge (5);
		\draw [->] (main) edge (3) (3) edge (4) (4) edge (5);
		
		\draw (main) -- ($(main) + (-45:.5)$); \draw (main) -- ($(main) + (-135:.5)$);
		
		\draw (2) -- ($(2) + (60:.5)$); \draw (2) -- ($(2) + (30:.5)$);
		
		\draw (3) -- ($(3) + (-120:.5)$); \draw (3) -- ($(3) + (-150:.5)$);
		\draw (4) -- ($(4) + (120:.5)$); \draw (4) -- ($(4) + (150:.5)$);
	\end{tikzpicture}\\
\caption{An Example of Rule 1}
\label{rule 1}
\end{figure}
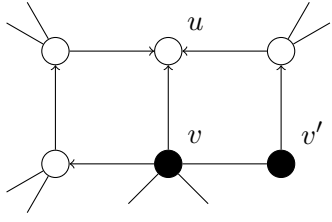

\begin{claim}\label{degree at most 4}If $u\in V(G)\backslash V(H)$ has degree at most four, $ch_*(u) \geq 3\gamma$.\end{claim}
\begin{proof}
First suppose $u$ is a 4-vertex.  Note that $u$ sends no charge under Rules 1-4.  By Lemma \ref{4-face with 4-vertices}, every 4-face $f$ incident to $u$ contains a major vertex $v_f$.
Therefore, if $u$ is adjacent to $k$ 4-faces, $u$ receives at least $\frac{3k\gamma}{4}$ charge under Rule 2.  By Rule 4, $u$ receives $\frac{3(4-k)\gamma}{4}$ charge from $5^+$-faces.  Hence, $u$ receives at least $3\gamma$ charge, as desired.

Therefore we may assume $u$ is a 3-vertex.  Note that $u$ sends no charge under Rules 1-4.  By Lemma \ref{4-face with 4-vertices}, every 4-face $f$ incident to $u$ contains a major vertex.  Hence, for every 4-face $f$ incident to $u$, there are two internally disjoint stamens $P_1$ and $P_2$ with an end at $u$ and an end at a major vertex such that every vertex in $P_1$ and $P_2$ is incident to $f$.  Note that a stamen is incident with at most two 4-faces.

Therefore, if $u$ is adjacent to $k$ 4-faces, $u$ receives at at least $\frac{k(1 + 3\gamma)}{3}$ charge under Rule 1.  By Rule 3, $u$ receives $\frac{(3 - k)(1 + 3\gamma)}{3}$ charge from $5^+$-faces.  Hence, $u$ receives at least $1 + 3\gamma$ charge, as desired.
\end{proof}

\begin{claim}\label{degree at least 7}If $v\in V(G)\backslash V(H)$ has degree at least seven and $\gamma\leq \frac{2}{13}$, $ch_*(v) \geq \frac{2}{3} - \frac{91\gamma}{4}$.\end{claim}
\begin{proof}
Let $P_1$ and $P_2$ be distinct $v$-stamens that are each incident with a 4-face.  Suppose $vv'\in E(P_1)\cap E(P_2)$.  Then $E(P_1)\cap E(P_2) = \{vv'\}$, and $P_1\triangle P_2$ is a $u$-stamen of length at most five, where $u$ is an end of $P_1$, contradicting Corollary \ref{path bt 3-vertices}.  Hence, $P_1$ and $P_2$ are internally disjoint.  Therefore $v$ sends charge at most $d(v)(\frac{1}{3}+\gamma)$ to 3-vertices under Rule 1.
Note that $v$ sends at most $d(v)\frac{9\gamma}{4}$ charge to 4-vertices under Rule 2.  Therefore $v$ sends charge at most $d(v)\left(\frac{1}{3} + \gamma + \frac{9\gamma}{4}\right)$.  Since $\gamma \leq \frac{2}{13},$

$$ch_*(v) \geq d(v) - 4 - d(v)\left(\frac{1}{3} + \gamma + \frac{9\gamma}{4}\right) = d(v)\left(\frac{2}{3} - \frac{13\gamma}{4}\right) - 4 \geq \frac{2}{3} - \frac{91\gamma}{4},$$
as desired.
\end{proof}

\begin{claim}\label{degree 6}If $v\in V(G)\backslash V(H)$ has degree six, then $ch_*(v) \geq \frac{2}{3} - \frac{35\gamma}{2}$.\end{claim}
\begin{proof}
Suppose $v$ sends charge at most $\frac{4}{3} + 4\gamma$ to 3-vertices under Rule 1.  Note that $v$ sends at most $d(v)\frac{9\gamma}{4} = \frac{27\gamma}{2}$ charge to 4-vertices under Rule 2.  Hence,

$$ch_*(v) \geq 2 - \left(\frac{4}{3} + 4\gamma\right) - \frac{54\gamma}{4} = \frac{2}{3} - \frac{35\gamma}{2},$$
as desired.

Therefore we may assume that $v$ sends greater than $\frac{4}{3} + 4\gamma$ charge to 3-vertices.  Then by Rule 1, there exist at least five $v$-stamens of $G$ $P_1,\dots, P_5$, where $u_i\neq v$ is an end of $P_i$, and each $P_i$ is incident with a 4-face, $f_i$.  Since $\epsilon \leq \frac{1}{5}$, by Corollary \ref{path bt 3-vertices}, the $P_i$ are pairwise internally disjoint.  Let $Q = \cup_{i=1}^4 P_i$.  We choose $P_1,\dots, P_5$ such that $(|V(P_1)|,\dots, |V(P_5)|)$ is lexicographically minimum over all $v$-stamens of $G$, and subject to that, $|V(Q)|$ is minimum.  Note that $Q$ is a poppy of $G$.  Since $\epsilon\leq\frac{1}{8}$, by Lemma \ref{super lemma}, $|V(Q)| > 8$.  Note that for all $i=1,\dots, 5$, $2\leq |V(P_i)|\leq 4$.  Furthermore, if $|V(P_i)| = 4$, then $v$ is adjacent to $u_i$, so there exists $j<i$ such that $u_j = u_i$ and $|V(P_j)| = 2$.

First we claim that $|V(P_2)| > 2$.  Suppose not.  Then $|V(P_1)| = |V(P_2)| = 2$.  If $|V(P_3)| = 3$, then since $v\in V(P_i)$ for all $i$, $|V(Q)|\leq 8$, a contradiction.  Therefore for $i=3,4,5$, $|V(P_i)| = 4$.  Since $|V(Q)|$ is minimum, $u_3$ is either $u_1$ or $u_2$.  Hence, $|V(Q)| \leq 8$, a contradiction.  Therefore $|V(P_2)| > 2$, as claimed.

We claim that $|V(P_1)| > 2$.  Suppose not.  Since $v\in V(P_i)$ for all $i$ and $|V(Q)| > 8$, $|V(P_4)| = 4$.  Since $|V(Q)|$ is minimum, $u_4 = u_1$.  Since $|V(Q)|\leq 8$, $|V(P_3)| = 4$.  Since $|V(P_2)| > 2$, $u_3 = u_1$.  Since $|V(Q)|\leq 8$, $|V(P_2)| = 4$.  Hence, $u_2 = u_1$, contradicting that $u_1$ has degree three.  Therefore $|V(P_1)| > 2$, as claimed.

Thus $|V(P_i)| > 2$ for all $i=1,\dots, 5$.  But then $|V(P_i)|\neq 4$ for all $i$.  Hence, $|V(P_i)| = 3$ for all $i=1,\dots, 5$.  Since $|V(Q)| > 8$ and $|V(Q)|$ is minimum, $u_1,\dots, u_5$ are distinct.  For each $i=1,\dots, 5$, let $w_i\in V(P_i)\backslash\{v,u_i\}$.  If there exists $i,j$ such that $i\neq j$ and $w_i$ is adjacent to $u_j$, then $u_iwu_j$ is a $u_i$-stamen, contradicting Corollary \ref{path bt 3-vertices}.  Therefore $w_1,\dots, w_5$ are distinct, and since the $u_1,\dots, u_5$ are distinct, $f_1,\dots, f_5$ are distinct.  But each $w_i$ is incident with at least two 4-faces that are incident to $v$.  Since $v$ is incident with at most six 4-faces, there exists some face $f$ incident to $v$ such that for all $i=1,\dots, 5$, $f\neq f_i$ and $w_i$ is incident with $f$.  Therefore for some $i\neq j$, $w_i=w_j$, a contradiction.
This completes the proof.\end{proof}

\begin{claim}\label{degree 5}If $v\in V(G)\backslash V(H)$ has degree five, then $ch_*(v) \geq \frac{1}{3} - \frac{53\gamma}{4}.$\end{claim}
\begin{proof}
Suppose $v$ sends charge at most $\frac{2}{3} + 2\gamma$ to 3-vertices under Rule 1.  Note that $v$ sends at most $d(v)\frac{9\gamma}{4} = \frac{45\gamma}{4}$ charge to 4-vertices under Rule 2.  Hence,

$$ch_*(v) \geq 1 - \left(\frac{2}{3} + 2\gamma\right) - \frac{45\gamma}{4} = \frac{1}{3} - \frac{53\gamma}{4},$$
as desired.

Therefore we may assume that $v$ sends greater than $\frac{2}{3} + 2\gamma$ charge to 3-vertices.  Then by Rule 1, there exist $v$-stamens $P_1, P_2,$ and $P_3$, where $u_i\neq v$ is an end of $P_i$, and each $P_i$ is incident with a 4-face, $f_i$.  Since $\epsilon \leq \frac{1}{5}$, by Corollary \ref{path bt 3-vertices}, the $P_i$ are pairwise internally disjoint.  

We choose $P_1,P_2,$ and $P_3$ such that $(|V(P_1),|V(P_2)|, |V(P_3)|)$ is lexicographically minimum over all $v$-stamens of $G$.  Let $Q=\cup_{i=1}^3P_i$.  Note that $Q$ is a poppy of $G$.  Since $\epsilon\leq\frac{1}{8}$, by Lemma \ref{super lemma}, $|V(Q)| > 8$.  Note that for all $i=1,2,3$, $2\leq |V(P_i)|\leq 4$.  Furthermore, if $|V(P_i)| = 4$, then $v$ is adjacent to $u_i$, so there exists $j<i$ such that $u_j=u_i$ and $|V(P_i)| = 2$.  Since $v\in V(P_i)$ for all $i$ and $|V(Q)| > 8$, $|V(P_1)| + |V(P_2)| + |V(P_3)| > 10$.  Since $|V(P_2)|,|V(P_3)|\leq 4$, $|V(P_1)| > 2$.  Hence, $|V(P_i)| = 3$ for all $i=1,2,3$.  Then $|V(Q)|\leq 7$, a contradiction.  This completes the proof.
\end{proof}

\begin{claim}\label{X major}If $v\in V(H)$ and $\gamma\leq\frac{2}{13}$, then $ch_*(v) \geq \min\{3\gamma,\frac{1}{3} - \frac{7\gamma}{2}\}$.
\end{claim}
\begin{proof}
If $v$ is a 1-vertex, then since $G$ is simple, $v$ is not incident to a 4-face unless $G$ is the path of length three.  Since $H$ is a proper subgraph of $G$, there is a vertex of degree at most two in $V(G)\backslash V(H)$, contradicting Corollary \ref{min degree 3}.  Therefore $G$ is not the path of length three, so $v$ is not incident to a 4-face.  Hence, $v$ sends no charge under Rules 1-4, so $ch_*(v)\geq 3\gamma$, as desired.  

Therefore we may assume $d(v)\geq 2$.  Since $\epsilon\leq\frac{1}{5}$, by Corollary \ref{path bt 3-vertices}, if $P_1$ and $P_2$ are distinct $v$-stamens that are each incident with a 4-face, then $P_1$ and $P_2$ are internally disjoint.  Therefore $v$ sends charge at most $d(v)(\frac{1}{3}+\gamma)$ to 3-vertices under Rule 1.
Note also that $v$ sends charge at most $d(v)\frac{9\gamma}{4}$ to 4-vertices under Rule 2.  Therefore,

$$ch_*(v) \geq d(v) + 3\gamma - 1 - d(v)\left(\frac{1}{3} + \gamma + \frac{9\gamma}{4}\right) 
= d(v)\left(\frac{2}{3} - \frac{13\gamma}{4}\right) + 3\gamma - 1 \geq \frac{1}{3} - \frac{7\gamma}{2},$$
as desired.
\end{proof}

\begin{claim}\label{faces}If $f\in F(G)$ and $\gamma \leq \frac{1}{15}$, then $ch_*(f)\geq 0$.\end{claim}
\begin{proof}
Let $f\in F(G)$.  If $|f| = 4$, then $f$ sends no charge under Rules 1-4.  Therefore $ch_*(f)\geq0$, as desired.

Suppose $|f|\geq 8$.  Under Rule 3, $f$ sends charge at most $|f|(\frac{1}{3} + \gamma)$ to 3-vertices.  Under Rule 4, $f$ sends charge at most $|f|\frac{3\gamma}{4}$ to 4-vertices.  Since $\gamma \leq \frac{1}{15}$, $f$ sends charge at most

$$|f|\left(\frac{1}{3} + \gamma + \frac{3\gamma}{4}\right) \leq \frac{27|f|}{60}  < \frac{1}{2}|f|.$$
Hence, $ch_*(f) \geq |f| - 4 - \frac{|f|}{2} = \frac{|f|}{2} - 4 \geq 0,$ as desired.

Suppose $5 < |f| < 8$.  By Corollary \ref{path bt 3-vertices}, since $\epsilon\leq\frac{1}{2}$, $G$ does not contain adjacent 3-vertices.  Therefore $f$ is incident to at most $\lfloor \frac{|f|}{2}\rfloor$ 3-vertices.  Since $G$ is triangle-free and $|f| < 8$, each 3-vertex appears at most once in the boundary walk of $f$.  Hence, $f$ sends charge at most $\frac{|f|}{2}\left(\frac{1}{3} + \gamma\right)$ to 3-vertices under Rule 3.  Under Rule 4, $f$ sends charge at most $|f|\frac{3\gamma}{4}$ to 4-vertices.  Therefore $f$ sends charge at most

$$\frac{|f|}{2}\left(\frac{1}{3} + \gamma\right) + |f|\frac{3\gamma}{4} = |f|\left(\frac{1}{6} + \frac{\gamma}{2} + \frac{3\gamma}{4}\right) = |f|\left(\frac{2 + 15\gamma}{12}\right).$$
Since $\gamma \leq \frac{1}{15}$, $f$ sends at most $\frac{|f|}{4}$ charge.  Hence, $ch_*(f) \geq |f| - 4 - \frac{|f|}{4} = \frac{3|f|}{4} - 4 \geq 0,$ as desired.

Suppose $|f| = 5$.  Since $G$ is triangle-free, each vertex appears at most once in the boundary walk of $f$.  If $f$ is not incident to any 3-vertices, then $f$ sends charge at most $5(\frac{3\gamma}{4}) \leq\frac{1}{4}$ under Rules 3 and 4, so $ch_*(f)\geq0$, as desired.  If $f$ is incident to precisely one 3-vertex, then $f$ sends charge at most $\frac{1}{3} + \gamma + 4(\frac{3\gamma}{4}) = \frac{1}{3} + 4\gamma \leq \frac{3}{5}$ under Rules 3 and 4, as desired.  If $f$ is incident to precisely two 3-vertices, then $f$ sends charge at most $\frac{2}{3} + 2\gamma + 3(\frac{3\gamma}{4}) = \frac{2}{3} + \frac{17\gamma}{4} \leq \frac{57}{60}$ under Rules 3 and 4, as desired.  Since $\epsilon\leq\frac{1}{2}$, $G$ does not contain adjacent 3-vertices by Corollary \ref{path bt 3-vertices}.  Hence, $f$ is incident to at most two 3-vertices, so the proof is complete.
\end{proof}

By Claims \ref{degree at most 4}, \ref{degree at least 7}, \ref{degree 6}, \ref{degree 5}, and 
\ref{X major}, if $\gamma\leq \frac{1}{15}$, then for all $v\in V(G)$, $ch_*(v) \geq \min\{3\gamma, \frac{2}{3} - \frac{91\gamma}{4}, \frac{1}{3} - \frac{53\gamma}{4}\}$.  So if $\gamma = \frac{4}{195}$, then $ch_*(v) \geq \frac{4}{65}$ for all $v\in V(G)$, and by Claim \ref{faces}, for all $f\in F(G)$, $ch_*(f)\geq0$.  Therefore

$$\frac{4}{65}|V(G)| \leq \sum_{v\in V(G)}ch_*(v) + \sum_{f\in F(G)}ch_*(f) = \left(\frac{199}{65}\right)|V(H)|+ 4(2g - 2).$$
Hence,

$$|V(G)| \leq \frac{199}{4}|V(H)| + 65(2g-2) \leq 50\left(|V(H)|-\frac{13}{5}\right) + 130g,$$
as desired.
\end{proof}



\begin{thebibliography}{99}

\bibitem{ADPT13} A. Asadi, Z. Dvo{\v{r}}{\'a}k, L. Postle, and R. Thomas, Sub-exponentially many 3-colorings of triangle-free planar graphs, J. Combin. Theory Ser. B, 103 (2013), 706--712.

\bibitem{BL46} G. D. Birkhoff, and D. C. Lewis, Chromatic polynomials, Trans. Amer. Math. Soc., 60 (1946), 355--451.

\bibitem{ERT80} P. Erd{\H{o}}s, A. L. Rubin, and H. Taylor, Choosability in graphs, Proceedings of the {W}est {C}oast {C}onference on {C}ombinatorics, {G}raph {T}heory and {C}omputing ({H}umboldt
              {S}tate {U}niv., {A}rcata, {C}alif., 1979), {Congress. Numer., XXVI}, {Utilitas Math., Winnipeg, Man.}, (1980), pp. 125--157.
              
\bibitem{P16} {L. Postle}, {3-List-Coloring Graphs of Girth at least Five: A Linear Isoperimetric Bound}, {manuscript}.

\bibitem{PT16} L. Postle and R. Thomas, Hyperbolic families and coloring graphs on surfaces, {manuscript}.

\bibitem{Th94} C. Thomassen, Every planar graph is {$5$}-choosable, J. Combin. Theory Ser. B, 62 (1994), 180--181.

\bibitem{Th03} C. Thomassen, A short list color proof of {G}r\"otzsch's theorem, J. Combin. Theory Ser. B, 88 (2003) 189--192.

\bibitem{Th06} {C. Thomassen}, {The number of {$k$}-colorings of a graph on a fixed surface}, {Discrete Math.}, 306 (2006), 3145--3153

\bibitem{Th07a} {C. Thomassen}, {Exponentially many 5-list-colorings of planar graphs}, {J. Combin. Theory Ser. B}, 97 (2007), 571--583.

\bibitem{Th07b} {C. Thomassen}, {Many 3-colorings of triangle-free planar graphs}, {J. Combin. Theory Ser. B}, 97 (2007) 334--349.

\end{thebibliography}


\end{document}